\documentclass{amsart}

\usepackage{amssymb, graphicx, amsrefs}
\usepackage{color}%cyan,red;magenta,green;yellow,blue

\copyrightinfo{2020}{Bobo Hua, Jun Masamune and Rados{\l}aw K. Wojciechowski}

\newtheorem{theorem}{Theorem}[]
\newtheorem{lemma}[theorem]{Lemma}

\newtheorem{corollary}[theorem]{Corollary}

\theoremstyle{definition}

\newtheorem{example}[]{Example}

\theoremstyle{remark}
\newtheorem{remark}[theorem]{Remark}

\numberwithin{equation}{section}

\newcommand{\as}[1]{\left\langle #1\right\rangle}

\newcommand{\ov}[1]{\overline{ #1}}

\newcommand{\oh}[1]{\widehat{ #1}}
\newcommand{\ms}[1]{\mathcal{ #1}}
\newcommand{\N}{\mathbb{N}}
\newcommand{\R}{\mathbb{R}}

\newcommand{\Om}{\Omega}
\newcommand{\cp}{\textup{cap}}
\newcommand{\vp}{\varphi}
\newcommand{\ve}{\varepsilon}
\newcommand{\si}{\sigma}
\newcommand{\vr}{\varrho}
\newcommand{\Deg}{\textup{Deg}}
\newcommand{\lm}{\lambda}

\providecommand{\eat}[1]{}

\newcommand{\Hm}[1]{\leavevmode{\marginpar{\tiny%
$\hbox to 0mm{\hspace*{-0.5mm}$\leftarrow$\hss}%
\vcenter{\vrule depth 0.1mm height 0.1mm width \the\marginparwidth}%
\hbox to 0mm{\hss$\rightarrow$\hspace*{-0.5mm}}$\\\relax\raggedright
#1}}}

\begin{document}

\title{Essential self-adjointness and the $L^2$-Liouville property}

\author{Bobo Hua}
\address{B.~Hua, School of Mathematical Sciences, Fudan University, 200433, Shanghai, China}
\address{Shanghai Center for Mathematical Sciences, Jiangwan Campus, Fudan University, No. 2005 Songhu Road, 200438, Shanghai, China.}
\email{bobohua@fudan.edu.cn}

\author{Jun Masamune}
\address{J.~Masamune, Department of Mathematics, Faculty of Science, Hokkaido University
Kita 10, Nishi 8, Kita-Ku, Sapporo, Hokkaido, 060-0810, Japan}
\email{jmasamune@math.sci.hokudai.ac.jp}

\author{Rados{\l}aw K. Wojciechowski}
\address{R.~K.~Wojciechowski, Graduate Center of the City University of New York, 365 Fifth Avenue, New York, NY, 10016}
\address{York College of the City University of New York, 94-20 Guy R. Brewer Blvd., Jamaica, NY 11451}

\email{rwojciechowski@gc.cuny.edu}

%\subjclass[2010]{Primary 39A12; Secondary 58J35}

\date{\today}
\thanks{{B.H.~is supported by NSFC, grants no.~11831004 and no.~11926313.}}
\thanks{J.M.~is supported in part by 
JSPS KAKENHI Grant Number 18K03290 and 17H01092.}
\thanks{R.W.~is supported by PSC-CUNY Awards, jointly funded by the Professional Staff Congress and the City University of New York, and the Collaboration Grant for Mathematicians, funded by the Simons Foundation.}

\begin{abstract}
We discuss connections between the essential self-adjointness of a symmetric operator and the constancy
of functions which are in the kernel of the adjoint of the operator.  
We then illustrate this relationship in the case of Laplacians on both manifolds and graphs. Furthermore, we discuss
the Green's function and when it gives a non-constant harmonic 
function which is square integrable.
\end{abstract} 
\maketitle
\tableofcontents

\section{Introduction}

Liouville theorems give conditions under which harmonic functions are constant.
For example, the classic Liouville theorem states that all bounded harmonic functions defined
on an entire Euclidean space are constant. 
In this note, we are interested
in the case when a harmonic function being square integrable forces the constancy of the function.
We call this the $L^2$-Liouville property. 

On the other hand, essential self-adjointness deals with the uniqueness of self-adjoint extensions
of a symmetric operator. Although seemingly unrelated to the $L^2$-Liouville property, 
for a strictly positive symmetric operator
there is a characterization of essential self-adjointness via the triviality of 
functions in the kernel of the adjoint of the operator which connects these properties. 
Abstractly, this connection is certainly known to experts, see, e.g., \cite{GM}.

The aim of this note is to first summarize this connection 
for symmetric operators on general $L^2$ spaces and then illustrate the statement for Laplacians
on both graphs and manifolds. In particular, while essential self-adjointness always implies
the $L^2$-Liouville property, additional assumptions, more precisely, positivity of the bottom of the spectrum
and the space having infinite measure, are needed for the converse implication.
We will use the graph and manifold case to show that these additional assumptions are necessary.

It is well known that the Laplacian on a complete Riemannian 
manifold is essentially self-adjoint \cite{C, S}
and that such manifolds satisfy the $L^2$-Liouville property \cite{S,Y}.
Analogous results have been more recently established for Laplacians on weighted graphs
via the use of intrinsic metrics \cite{HK, HKMW}.
For an incomplete Riemannian manifold such as
$M= N \setminus K$ where $N$ is complete and $K$ is a compact submanifold,
the Laplacian is essentially self-adjoint if and only if the codimension of $K$ is greater than or equal to 4, 
see \cite{CdV, M}.
We mention \cite{BP} for further results concerning essential self-adjointness of 
the Laplacian on an incomplete manifold and \cite{I} for some results in the graph case.
In contrast, the $L^2$-Liouville property holds on $\mathbb R^n \setminus \{p\}$ for any dimension \cite{ABR}, while
it fails on $\mathbb H^n \setminus \{p\}$ for $n=2, 3$ as the Green's function on $\mathbb H^n \setminus \{p\}$ 
is in $L^2$ for $n=2, 3$ as we will discuss below.
Furthermore, we will give some analogues for the incomplete case for graphs by considering a graph
with a vertex removed and  with paths of finite length attached at neighboring vertices of the removed vertex.

We organize the paper as follows. 
In Section \ref{Preliminary results} we use general theory to point out the connection between essential self-adjointness 
and the $L^2$-Liouville property.
In Section \ref{s:examples} we first recall the setting of Laplacians on weighted graphs and
Riemannian manifolds
and then show by examples that the conditions in the general results obtained in 
Section \ref{Preliminary results} are necessary.
In Section \ref{Liouville property manifolds} we study the $L^2$-Liouville property on a manifold with a point removed.
More precisely, we discuss how the dimension being less than 4 ensures that the Green's function is
square integrable on a neighborhood of the pole and positivity of the bottom of the spectrum 
implies that the Green's function is square integrable outside of this neighborhood.
It should be mentioned that some results in this section can be obtained by combining our general theory with 
known results concerning essential self-adjointness of the Laplacian on manifolds.
In Section \ref{s:Greens_graphs} we extend the ideas of the previous section on manifolds
to weighted graphs.

Finally, let us note that, in both the manifold and graph settings, 
there are known connections between Liouville theorems for bounded
functions and recurrence, see \cite{G, W}, and between the $L^1$-Liouville property 
and stochastic completeness, see \cite{G, HK}. Hence, Liouville theorems and stochastic 
properties are connected.
In this note we add another piece to this picture in pointing out the connection between the $L^2$-Liouville
property and essential self-adjointness.

\section{Preliminary results}\label{Preliminary results}

We start with some standard Hilbert space notions. For more details, see,
for example, Sections VIII.1 and VIII.2 in \cite{RS1} and Section X.3 in \cite{RS2}.
We let $\ms{H}$ denote a Hilbert space.
We let 
$$T:D(T)\subseteq \ms{H} \longrightarrow \ms{H}$$
denote a symmetric operator with dense domain $D(T)$, adjoint $T^*$
and closure $\ov{T}=T^{**}$.
We say that $T$ is \emph{essentially self-adjoint} if $\ov{T}$ is self-adjoint,
equivalently, if $T$ has a unique self-adjoint extension.
We say that $T$ is \emph{strictly positive} if there exists a constant $C>0$ such that
$$\as{Tf, f} \geq C \|f\|^2$$ 
for all $f \in D(T)$.

We now highlight the other main property of interest which concerns the constancy of certain functions.    
In the case when $\ms{H}=L^2(X,m)$ for a measure space $(X,m)$
we say that $T$ satisfies the $L^2$-\emph{Liouville property} if every function $f \in \ker T^*$ is constant.   
In particular, we note that any function in the kernel of a self-adjoint extension of $T$ will
be constant in this case.

For the examples that we have in mind, namely Laplacians on graphs and manifolds without boundary,
it turns out that the $L^2$-Liouville property is equivalent to the fact that every function which is harmonic
in a pointwise or distributional sense and is in the corresponding $L^2$ space is constant.
We now highlight the connection between essential self-adjointness and the $L^2$-Liouville property
introduced above.
\begin{theorem}\label{Main_theorem}
Let $(X,m)$ be a measure space and let $T:D(T) \subseteq L^2(X,m) \longrightarrow L^2(X,m)$
be a symmetric operator such that every function $f \in \ker \overline{T}$ is constant.
\begin{itemize}
\item[(1)] If $T$ is essentially self-adjoint, then $T$ satisfies the $L^2$-Liouville property.
\item[(2)]  If $T$ satisfies the $L^2$-Liouville property, $T$ is strictly positive and 
$m(X)=\infty$, then $T$ is essentially self-adjoint.
\end{itemize}
\end{theorem}
\begin{proof}
(1)  Since the adjoint is a closed operator, essential self-adjointness implies 
$T^* = \overline{T^*}= {\overline{T}}^*=\overline{T}$.
Therefore, if $f \in \ker T^*$, then $f \in \ker \overline{T}$ and thus is constant by the 
assumption on $T$. Hence, $T$ satisfies the $L^2$-Liouville property.

(2)  We note that by Theorem X.26 in \cite{RS2} and the strict positivity assumption, 
$T$  is essentially self-adjoint if and only if $\ker T^* = \{0\}$.
Let $f \in \ker T^*$. Then, as $T$ has the $L^2$-Liouville property, it follows that $f$ is constant.
Since $f \in L^2(X,m)$ and $m(X)=\infty$, it now follows that $f$ is trivial.
This completes the proof.
\end{proof}

We now comment on the assumption that every $f \in \ker \overline{T}$ is constant found in
the result above. First we note that this is only used in the proof of statement (1). However, the 
strict positivity of $T$ assumed in (2) automatically implies that any function in $\ker \overline{T}$
is actually trivial.
Furthermore, in the case of Laplacians on graphs and manifolds that are the main focus of this paper, 
$f \in \ker \overline{T}$ implies that $f$ has zero energy as we will discuss below. 
Thus, $f  \in \ker \overline{T}$ is constant in the case that the underlying space is connected.

%%%%%%%%%%%%%%%%%%%%%%%%%%%%%%%%%%%%%%%%%%%%%%%%%%%%%%%%%%%%%%
\section{Laplacians on graphs and manifolds}\label{s:examples}
We will now illustrate the general results from the previous section with Laplacians on weighted graphs and 
manifolds. In particular, we will use these examples to illustrate the necessity of the additional assumptions
needed for the $L^2$-Liouville property to imply essential self-adjointness.

\subsection{Laplacians on weighted graphs}\label{ss:Lap_weighted_graphs}

We first introduce the setting of infinite weighted graphs and 
the associated Laplacians as in \cite{KL}.  
Let $X$ be a countably infinite set of vertices with $b:X\times X \longrightarrow [0,\infty)$ 
satisfying $b(x,x)=0$ for all $x \in X$, $b(x,y)=b(y,x)$ for all $x, y \in X$ and $\sum_y b(x,y)<\infty$ for all $x \in X$.  
When $b(x,y)>0$ we think of $x$ and $y$ as neighbors connected by an edge with weight $b(x,y)$ and write $x \sim y$.
We note that we allow vertices to have infinitely many neighbors.
Let $m:X \longrightarrow (0,\infty)$ and extend $m$ to all subsets of $X$ by additivity.  
We call $b$ a graph over the measure space $(X,m)$.
In this case, our Hilbert space is 
$$\ms{H} = \ell^2(X,m) = \{ f: X \longrightarrow \mathbb{R} \ | \ \sum_{x \in X}f^2(x)m(x)<\infty\}$$
with inner product $\langle f, g \rangle =\sum_{x \in X} f(x)g(x)m(x)$.  
We denote the corresponding norm by $\|f\|^2=\langle f, f \rangle.$

Let $C(X)=\{ f: X \longrightarrow \mathbb{R} \}$ and let
$$\ms{F} = \{f \in C(X) \ | \ \sum_{y \in X} b(x,y) |f(y)| <\infty \textup{ for all } x \in X \}.$$ 
For a function $f \in \ms{F}$, we define the formal Laplacian $\ms{L}$ as
$$\ms{L}f(x) = \frac{1}{m(x)} \sum_{y \in X} b(x,y) (f(x) - f(y))$$
for $x \in X$.
We say that a function $f \in \ms{F}$ is \emph{harmonic} if 
$$\ms{L}f=0.$$
We note that this definition does not involve any Hilbert space 
and that harmonicity does not depend on the choice of the measure $m$.

The measure starts to play a role as soon as we require the function to additionally
be in $\ell^2(X,m)$.
We now point out a case in which all harmonic functions are trivial.
This was already observed in \cite{HK}, see also \cite{KL, Sch2, Woj} for similar reasoning.

\begin{example}[Infinite measure of paths]\label{Ex:measure}
We call a sequence of vertices $(x_n)$ an infinite path if $x_n \sim x_{n+1}$
for all $n \in \N$. If every infinite path has infinite measure, i.e., 
$$\sum_n m(x_n)=\infty \quad \textup{ for all infinite paths } (x_n), $$
then any harmonic function in $\ell^2(X,m)$ is trivial. This follows, as, by a direct calculation and induction, any harmonic function
which is not constant must strictly increase over some infinite path. Since this path will have
infinite measure, it follows that such a function cannot be in $\ell^2(X,m)$.

On the other hand, it is clear that
any non-trivial constant function cannot be in $\ell^2(X,m)$ if $X$ has infinite measure. 
Thus, we see that in the case of infinite measure of infinite paths, all harmonic functions
which are in $\ell^2(X,m)$ are trivial.
In particular, the only case of interest for Liouville properties involving $\ell^2(X,m)$ 
is the case when the measure decays so that some path has finite measure.
\end{example}

We let $C_c(X)$ denote the finitely supported functions in $C(X)$ and let $L_c$ denote the restriction
of $\ms{L}$ to $C_c(X)$, that is, $D(L_c)=C_c(X)$ and $L_c f = \ms{L}f$ for all $f \in C_c(X)$.
In order for $L_c$ to be symmetric, we assume throughout that 
$$\ms{L}(C_c(X)) \subseteq \ell^2(X,m).$$
We then say that $L_c$ is essentially self-adjoint if $L_c$ has a unique self-adjoint extension.
When $\ms{L}(C_c(X)) \subseteq \ell^2(X,m)$, one can calculate that 
$$D(L_c^*)=\{ f\in \ell^2(X,m) \cap \ms{F} \ | \ \ms{L}f \in \ell^2(X,m) \}$$
and that $L_c^*$ is a restriction of $\ms{L}$ to $D(L_c^*)$, see the proof of Theorem~6 in \cite{KL}.

We note that $f \in \ker(L_c^*)$ if and only if $f \in \ell^2(X,m) \cap \ms{F}$ and $\ms{L}f=0$.
In particular, $L_c$ satisfies the $\ell^2$-Liouville property if and only if all harmonic functions
which are in $\ell^2(X,m)$ are constant.

Example~\ref{Ex:measure} shows that under the
assumption that infinite paths have infinite measure $L_c$ satisfies the $\ell^2$-Liouville property.
Furthermore, under this assumption, $L_c$ is essentially self-adjoint by Theorem~6~in~\cite{KL}.

Having introduced the two main properties of interest for the Laplacian in the graph setting we now turn
to the form perspective. This will be used to show that functions which are in the kernel
of $\overline{L}_c$ are constant when the graph is connected.
Given the symmetric operator $L_c$ we can define the associated form $Q_c$ via
$$Q_c(f,g)=\as{L_c f, g}$$
for all $f, g \in D(L_c)$.
A direct calculation gives that the form $Q_c$ is a restriction of the energy form $\ms{Q}$
which is given by
$$\ms{Q}(f,g) = \frac{1}{2}\sum_{x, y \in X} b(x,y) (f(x)-f(y))(g(x)-g(y))$$
for $f, g \in \ms{D}$ where $\ms{D}$ is the space of functions of finite energy defined as 
$$\ms{D}= \{ f\in C(X) \ | \ \sum_{x, y \in X} b(x,y) (f(x)-f(y))^2 <\infty\}.$$
A direct calculation gives the following version of a Green's formula
$$\langle L_c f, g \rangle = Q(f,g) = \ms{Q}(f,g) = \langle f, L_c g \rangle$$
for all $f, g \in D(L_c)$. In particular, we note that $L_c$ is positive and the densely defined form $Q_c$ is closable.
We denote the closure of $Q_c$ via $Q$ and note that $Q$ is a restriction
of $\ms{Q}$.

We now assume that $b$ is connected, that is, for $x, y \in X$, there exists a sequence
of vertices $(x_i)_{i=0}^n$ such that $x_0=x$, $x_n=y$ and $b(x_i,x_{i+1})>0$ for all $i=0, 1, \ldots, n-1$.
If $b$ is connected and $f \in \ker{\overline{L}_c}\subseteq D(Q)$, then it follows that 
$$Q(f) = \frac{1}{2}\sum_{x, y \in X}b(x,y)(f(x)-f(y))^2 =0$$
so that $f$ is constant.
In particular,
Theorem~\ref{Main_theorem} applies and gives the following connection in 
the setting of weighted graphs.
\begin{corollary}\label{c:graphs_esa_liouv}
Let $b$ be a connected graph over $(X,m)$ with $\mathcal{L}(C_c(X)) \subseteq \ell^2(X,m)$ 
and let $L_c$ be the restriction of the formal Laplacian to $C_c(X)$.
\begin{itemize}
\item[(1)] If $L_c$ is essentially self-adjoint, then $L_c$ satisfies the $\ell^2$-Liouville property, i.e., 
all harmonic functions in $\ell^2(X,m)$ are constant.
\item[(2)] If $L_c$ satisfies the $\ell^2$-Liouville property, $L_c$ is strictly positive and 
$m(X)=\infty$, then $L_c$ is essentially self-adjoint.
\end{itemize}
\end{corollary}

\begin{remark}
We note, in particular, the following consequence of (1) from Corollary~\ref{c:graphs_esa_liouv}.
A function $f \in \ms{F}$ is called $\lm$-harmonic for $\lm \in \R$ if
$$\ms{L}f=\lm f.$$
Clearly, when $\lm=0$, this just means that $f$ is harmonic. However, we note that
while the measure plays no role in the definition of a harmonic function, it does so for 
$\lm$-harmonic functions for $\lm \neq 0$. Now, by abstract theory,
it can be shown that $L_c$ is essentially self-adjoint if and only if every $\lm$-harmonic 
function for $\lm<0$ which is in $\ell^2(X,m)$ is trivial, see, for example \cite{Woj, KL} or
Theorem~6.2 in \cite{HKLW}. Hence, by taking the contrapositive of (1), we
see that the existence of a non-constant harmonic function in $\ell^2(X,m)$ implies
the existence of a non-trivial $\lm$-harmonic function in $\ell^2(X,m)$ for $\lm<0$.
\end{remark}

We will now discuss why the additional assumptions of strict positivity and infinite measure
are needed for the implication in (2) above.
In order to discuss strict positivity, we introduce the bottom of the spectrum of the Friedrichs
extension of $L_c$.
We denote the self-adjoint operator associated to the closed form $Q$ by $L$ and refer
to $L$ as the Laplacian associated to $b$ over $(X,m)$. 
We note that $\overline{L}_c \subseteq L$ and these operators
are equivalent precisely when $L_c$ is essentially self-adjoint.

We note from the spectral  theorem that $\lambda_0(L)$, the bottom of the spectrum of $L$, can be given as
$$\lambda_0(L) = \inf_{f \in D(L), f \neq 0} \frac{\langle Lf,f\rangle}{\langle f, f \rangle}= 
\inf_{\vp \in C_c(X), \vp \neq 0} \frac{\langle L_c \vp,\vp\rangle}{\langle \vp, \vp \rangle}.$$
In particular, the strict positivity of $L_c$ is equivalent to the positivity of the bottom of the spectrum of $L$, i.e.,
$\lambda_0(L)>0$.

We further note that $L$ which comes from $Q$ is not the only natural self-adjoint extension of $L_c$.
Namely, one can also consider the Neumann Laplacian $L_N$ which arises from the restriction
of the energy form $\ms{Q}$ to 
$$H^1=\{ f \in \ell^2(X,m) \ | \ \ms{Q}(f)<\infty \}.$$
We let $H_0^1 = D(Q)$ which is the closure of $C_c(X)$ with respect to the form norm
$\|f\|_{\ms{Q}}^2 = \|f\|^2 + \ms{Q}(f)$, i.e.,
$$H_0^1 = \ov{C_c(X)}^{\| \cdot \|_{\ms{Q}}}.$$ 
If $H_0^1 \neq H^1$, which is referred to as the failure of form uniqueness, then there are two distinct
self-adjoint extensions of $L_c$, namely the Laplacian $L$ and the Neumann Laplacian $L_N$, 
and thus $L_c$ is not essentially self-adjoint.

We now introduce a number of analytic and geometric quantities which will be needed for the examples
below. For more details, see \cite{HKMW}.
Specifically, whenever $\varrho$ is a strongly intrinsic path metric and $\ov{X}$ is the metric completion
of $X$ with respect to $\varrho$, then we let $\partial X = \ov{X} \setminus X$ denote the Cauchy 
boundary of $X$. 
For a set $K \subseteq X$, we let 
$$\mbox{Cap}(K)= \inf \{ \|f\|_{\ms{Q}} \ | \  f \in H^1, f \geq 1 \mbox{ on } K\}$$
denote the capacity of $K$ and let 
$$\mbox{Cap}(\partial X) = \inf \{ \mbox{Cap}(K \cap X) \ | \ \partial X \subseteq K \mbox{ with } K \subseteq \ov{X}
\mbox{ open} \}.$$
It is known that if $\mbox{Cap}(\partial X) < \infty$, then
$\mbox{Cap}(\partial X)>0$ if and only if $H^1_0 \neq H^1$, see Theorem 3 in \cite{HKMW}.
We note that in the case that every infinite path has infinite measure as discussed
in Example~\ref{Ex:measure} above, it is not hard to see that if $\partial X \neq \emptyset$, then
$\mbox{Cap}(\partial X)=\infty.$

We now show the necessity of assuming both infinite measure
and strict positivity in order for the $\ell^2$-Liouville property to imply essential self-adjointness.

\begin{example}[Necessity of additional assumptions in (2) of 
Corollary~\ref{c:graphs_esa_liouv}]\label{e:graphs_add} \ 

(i) Necessity of $m(X)=\infty$. 

Consider the graph with $X =\mathbb{N}=\{1, 2, 3, \ldots\}$ 
and $b(x,y)>0$ if and only if $|x-y|=1$.  
It is easy to see that if $\ms{L}f = 0$, then $f$ is constant. 
In particular, $L_c$ has the $\ell^2$-Liouville property
regardless of the choice of the measure $m$.

On the other hand, if
$$m(X)<\infty \qquad \textup{ and } \qquad \sum_{n=1}^\infty \frac{1}{b(n,n+1)}<\infty$$
it turns out that $L_c$ is not essentially self-adjoint. 
This can be seen as follows:
we consider $\lm$-harmonic functions for $\lm<0$, i.e., functions satisfying $\ms{L}f=\lambda f$
for $\lm<0$. As mentioned above, the essential self-adjointness of $L_c$
is equivalent to the triviality of such functions in $\ell^2(X,m)$.
Now, such a function $f$ is in $\ell^2(X,m)$ if $m(X)<\infty$ and $f$ is bounded. 
If $f$ is non-trivial and $m(X)<\infty$, then $f$ is bounded if and only if
$$\sum_{n=1}^\infty \frac{1}{b(n,n+1)}<\infty$$
as can be seen by applying Lemma 5.4 from \cite{KLW}.  
Therefore, in the case that $m(X)<\infty$ and $\sum_n 1/b(n,n+1)<\infty$,
we have a non-trivial $\lm$-harmonic function for $\lm<0$ which is in $\ell^2(X,m)$
and thus $L_c$ is not essentially self-adjoint by Theorem~6.2 in \cite{HKLW}. This shows the 
necessity of $m(X)=\infty$ in (2) of Corollary~\ref{c:graphs_esa_liouv}.

We now make some further remarks which will be used for the second example directly below.
In particular, we want to establish that the Cauchy boundary of $X$ has positive and finite capacity.
As $m(X)<\infty$ and the graph is stochastically incomplete by Theorem~5 in \cite{KLW}, 
it follows that $H^1\neq H^1_0$ as stochastic incompleteness, transience and $H^1 \neq H^1_0$
are all equivalent in the case of finite measure by results of \cite{Sch}, see also \cite{GHKLW}.
As the capacity of the Cauchy boundary is finite since the measure
of the graph is finite, Theorem 3 in \cite{HKMW} gives that
$$0<\mbox{Cap}(\partial X)<\infty.$$
This will be used in the next example directly below.

We further note that $\lambda_0(L)>0$ in this case by Theorem~3 in \cite{KLW} so that 
$L_c$ is strictly positive. To summarize:
$L_c$ is a strictly positive operator which satisfies the $\ell^2$-Liouville property
but is not essentially self-adjoint. \\

(ii) Necessity of strict positivity. 

We build on the example above. Specifically, we let $X=X_1 \cup X_2$ where
$X_1 =\mathbb{N}=\{1, 2, 3, \dots\}$ and $b(x,y)>0$ for $x, y \in X_1$ if and only if $|x-y|=1$ with 
$m(X_1)<\infty$ and $\sum_n 1/b(n,n+1)<\infty$ so that this is a graph as in (i)
above. 
We let $X_2= -\mathbb{N}=\{-1, -2, -3, \ldots\}$
with $b(x,y)=1$ for $x, y \in X_2$ if and only if $|x-y|=1$ and zero otherwise and $m$ satisfying
$m(X_2)=\infty$. Finally, we let $b(-1,1)=b(1,-1)>0$
so that the resulting graph is connected and add no other edges.

We now show that $L_c$ has the $\ell^2$-Liouville property. We first note that if $f \in \ms{F}$ satisfies $\ms{L}f=0$
and there exists a pair of vertices $x,y \in X$ such that $x \sim y$ and $f(x)\neq f(y)$, then 
$f(x) \neq f(y)$ for all $x, y \in X$ with $x \sim y$. Hence, if $f$ is harmonic and not constant, then it is not 
equal on all neighbors. Now, suppose that $f$ is harmonic, $f(-1) < f(-2)$ and let $C = f(-2)-f(-1)>0$. It follows by induction 
using $\ms{L}f=0$ that $f(-n) = f(-n+1)+C$ so that
$$f(-n) = f(-2)+C \cdot (n-2) $$ 
for all $n \geq 2$. 
In particular, $f$ will not be bounded above on $X_2$ and thus $f \not \in \ell^2(X,m)$ since $m(X_2)=\infty$. 
A similar argument works if $f(-1) > f(-2)$.
Therefore, all harmonic functions which are in $\ell^2(X,m)$ are constant and, furthermore, trivial
as $m(X)=\infty$. In particular,
$L_c$ has the $\ell^2$-Liouville property.

We now show that $L_c$ is not strictly positive.
By letting $1_n = 1_{K_n}$ denote the characteristic function of the set
$K_n=\{-1, -2, \ldots, -n\}$ we see that
$$\lambda_0(L) \leq \frac{\langle L_c 1_n,1_n\rangle}{\langle 1_n, 1_n \rangle}= \frac{2}{m(K_n)} \to 0$$
as $n \to \infty$. Therefore, $L_c$ is not strictly positive. 

Finally, by applying Theorem~3 from
\cite{HKMW}, we see that $L_c$ is not essentially self-adjoint. More specifically, as the usual combinatorial graph
metric is equivalent to a strongly intrinsic path metric on $X_2$,
the Cauchy boundary of $X$ is equal to the Cauchy boundary of $X_1$, that is,
$\partial X = \partial X_1$. In particular, since the Cauchy boundary of $X_1$ has finite and positive
capacity as discussed in (i) above, it follows that $H_0^1 \neq H^1$ and thus $L_c$ is not essentially self-adjoint. 

We further note that $m(X)=\infty$ in this case. To summarize:
$L_c$ satisfies the $\ell^2$-Liouville property
and the space has infinite measure but $L_c$ is not essentially self-adjoint.
\end{example}

%%%%%%%%%%%%%%%%%%%%%%%%%%%%%%%%%%%%%%%%%%%%%%%%%%%%%%%%%%%%%%
 
\subsection{Laplacians on manifolds}
We now introduce Laplacians on Riemannian manifolds. For further background
see, e.g., \cite{G19}.  
We let $(M,g)$ be a smooth connected Riemannian manifold.
We consider the measure space $(M,\mu)$ with the Riemannian measure $d\mu$ associated to $g$.
Our Hilbert space is $\mathcal H = L^2(M, \mu)= \{u \in L^0 \mid \|u\|<\infty\}$ where $\|u\|$ is induced by
the inner product
\[
\langle u,v \rangle = \int_M uv \, d\mu
\]
for $u,v \in L^2(M, \mu)$.
The Laplacian is defined as a distribution via $\Delta u = -\mbox{div} \nabla u$.
Here, $\nabla$ is defined by
\[g(\nabla u, \xi) (x)= \xi u (x) \]
for all $x \in M$ and all $\xi \in \Gamma(TM)$ where
$ \Gamma(TM)$ is the space of all smooth sections of the tangent bundle $TM$ and
$ \mbox{div}:\Gamma(TM)\longrightarrow  \Gamma(TM)$ is defined by
\[\langle \mbox{div} \xi, \psi \rangle = \int_M g(\xi, -\nabla \psi)\, d\mu\]
for all $\psi \in C_c^\infty(M)$ where $C_c^\infty(M)$ is the space of smooth functions with compact support on $M$.
If $M$ has a boundary, then we always assume that $M$ is orientable and impose Neumann boundary conditions for div.

A generalized function $u$ is called \emph{harmonic} if $u$ is a distribution and $\Delta u =0$ as a distribution.
We note that by the hypo-ellipticity of $\Delta$, any harmonic function is smooth.
We denote by $T$ the operator acting as $\Delta$ on the 
domain $C_c^\infty(M)$.
We note that if $M$ has a boundary, then $T$ satisfies Neumann boundary conditions.
By Green's formula,
\[ \langle u,\Delta v \rangle = \int_M g(\nabla u, \nabla v)\, d\mu = \langle \Delta u, v \rangle\]
for $u,v \in D(T)$. Therefore, $T$ is symmetric. 

On any Riemannian manifold, there are two self-adjoint extensions of $T$, namely,
$\Delta_D = \nabla_D^* \nabla_D$,
the Dirichlet Laplacian, and $\Delta_N =\nabla_N^*\nabla_N$, the Neumann Laplacian.
The domains of these operators are
\begin{align*}
 D(\nabla_D) &=H^1_0= \mbox{the closure of $C_c^\infty$ in $H^1$}  \\
 D(\nabla_N)&=H^1=\{u \in L^2 \mid \nabla u \in \vec L^2\}
\end{align*}
where $H^1$ is a Hilbert space with $\langle u,v \rangle_{1} = \langle u,v \rangle + \langle \nabla u, \nabla v \rangle$. 
Furthermore, $T$ is strictly positive if and only if the bottom of the spectrum of $\Delta_D$ is positive, that is,
\[\lambda_0 (\Delta_D) = \inf_{u \in H^1_0(M),\ u \neq 0} { \|\nabla u\|^2 \over \|u\|^2 } > 0.\]
We consider the energy form $a$ with domain $H^1_0$ and acting as
\[a(f) = \int_M g(\nabla f, \nabla f)\, d\mu\]
for $f \in H^1_0$.
Let $f$ be in the kernel of $\overline{T}$ and let $(f_n)$ be a sequence in $C_c^\infty(M)$ such that $f_n \to f$ and 
$T f_n \to \overline T f$ in $\mathcal H$.
Then, by the Green's formula above and by lower semicontinuity of $a$ which follows as $a$ is
a closed form by definition, we obtain
\[
0=\langle \overline T f, f \rangle=
\lim_{n\to\infty} \langle T f_n, f_n \rangle = \liminf_{n\to\infty} a(f_n) \ge a(f).
\]
Hence, $\nabla f=0$, and by the weak Poincar\'e inequality, see \cite{CL}, 
$f$ is constant.
In particular, all functions in the kernel of $\overline{T}$ are constant and thus 
Theorem~\ref{Main_theorem} applies and gives the following connection in 
the setting of Riemannian manifolds.
\begin{corollary}\label{manifold_corollary}
Let $M$ be a connected Riemannian manifold. 
Let $T$ be the Laplacian with domain $C^\infty_c(M)$.
\begin{itemize}
\item[(1)] If $T$ is essentially self-adjoint, then $T$ satisfies the $L^2$-Liouville property.
In particular, all harmonic functions in $L^2(M,\mu)$ are constant if $M$ has no boundary.
\item[(2)] If $T$ satisfies the $L^2$-Liouville property, $T$ is strictly positive and 
$\mu(M)=\infty$, then $T$ is essentially self-adjoint.
\end{itemize}
\end{corollary}
We note that Gaffney showed that if $M$ is complete, then $H^1_0 = H^1$ \cite{G1, G2}. 
We remark that the condition $H^1_0 = H^1$ does not imply the 
essential self-adjointness of $T$. For example, if $M = N \setminus K$
where $N$ is a complete Riemannian manifold with dimension $n$
and $K$ is a compact submanifold with dimension $k$,
then $H^1_0= H^1$ if and only if $n-k \ge 2$ and
$T$ is essentially self-adjoint  if and only if $n-k \ge 4$, see \cite{CdV, M}. 
Clearly, if $H^1_0 \neq H^1$,  then $\nabla_D \neq \nabla_N$ so that $T$ is not essentially self-adjoint.

We let $\overline{M}$ denote the metric completion of $M$
and call $\partial M  = \overline{M} \setminus M$ the Cauchy boundary of $M$.
We recall that the capacity of $\partial M$ is given by
\[\mbox{Cap}(\partial M) = \inf \{ \|u \|^2_1 \mid u \in H^1, u=1
\mbox{ on a neighborhood of $\partial M$ in $M$}  \}. \]
It is known that if $0<\mbox{Cap}(\partial M)<\infty$, then $H^1_0 \neq H^1$, see \cite{M}.

We now show the necessity of the additional assumptions in statement (2) of 
Theorem~\ref{Main_theorem} in this continuous model.
 
\begin{example}[Necessity of additional assumptions in (2) of Theorem~\ref{Main_theorem}]
\ 

(i) Necessity of strict positivity. 

To show the necessity of $T$ being strictly positive in order for the $L^2$-Liouville property to imply essential
self-adjointness, we
consider $T$ on the half-line $(0,\infty) \subset \mathbb R$.
We observe that $T$ is not strictly positive. Indeed, let
\[
u_n (x) =
\begin{cases}
0, & x \in (0,1),\\
{1 \over n^3}(x-1), &x \in [1,n^2+1], \\
{1 \over n^3}((2n^2+1)-x), &x \in [n^2+1,2n^2+1],\\
0, & x > 2n^2+1.
\end{cases}
\]
Then $u_n \in H^1_0 ((0,\infty))$ and
\[
\lambda_0(\Delta) \le { \|\nabla u_n\|^2 \over \| u_n\|^2 } \to 0
\]
as $n \to \infty$. Thus, $T$ is not strictly positive.

The Cauchy boundary of $(0,\infty)$ is $\{0\}$.  
If $\mbox{Cap}(0) =0$, then there exists a sequence $(f_n)$ in $H^1$  
and sequences $(x_n)$ and $(y_n)$ in $(0,\infty)$ such that 
 $0< x_n < y_n$, $y_n\to0$ as $n\to\infty$ with
 \[
 \begin{cases}
 f_n|_{(0,x_n)}=1 \\
 \Delta f_n|_{(x_n,y_n)} =0 \\
 f_n|_{(y_n,\infty)}=0 
\end{cases}
 \]
and 
\[ 0 = \lim_{n\to\infty} \int_{x_n}^{y_n} |f'_n(x)|^2\, dx.\]
Since $f_n(x) = {1 \over x_n-y_n} (x-y_n)$ for $x \in (x_n,y_n)$, we get
\[
\int_{x_n}^{y_n} |f'_n(x)|^2\, dx =
 {1 \over y_n-x_n}
\]
which clearly does not tend to 0. This contradiction implies $\mbox{Cap}(0) >0$
so that $H^1_0 \neq H^1$ and thus
$T$ on $(0,\infty)$
is not essentially self-adjoint.

However, since any harmonic function
 $f$ on $(0,\infty)$, i.e., any function satisfying $f\cprime \cprime =0$, is affine, 
in order for $f$ to be in $L^2$, $f$ must be equal to 0. This gives
the $L^2$-Liouville property. \\

(ii) Necessity of  $\mu(X)=\infty$. 

To show the necessity of $\mu(X)=\infty$, we
consider the interval $(0,1] \subset \mathbb R$.
This is a manifold with boundary $\{1\}$ so that
we impose the Neumann boundary condition at $\{1\}$. 
\eat{That is, the domain of $T$ is
\[D= \{u \in C^\infty_c ((0,1])\mid u'(1)=0 \}\]
and $T$ acts as $\Delta$.}
Then $T$ satisfies the assumptions of Theorem \ref{Main_theorem} and
is strictly positive. In fact,
for $u \in D(T)=C_c^\infty((0,1])$ we have
\begin{align*}
\|u\|^2_{L^2}&=\int^1_0|u(x)|^2\, dx  \\
&= \int^1_0 \left( \int^x_0 |u'(s)|\, ds \right)^2 dx \le \int^1_0 dx \int^1_0 |u'(s)|^2\, ds
=\|u'\|^2_{L^2}.
\end{align*}
Moreover, $T$ is not essentially self-adjoint by the same argument as in (i).
However, since any harmonic function $f$ on $(0,1]$ is affine, 
in order for $f$ to be in $D(T^*)$, $f$ must be a constant function 
by the Neumann boundary condition at $\{1\}$. 
This implies that $T$ satisfies the $L^2$-Liouville property.
\end{example}

%%%%%%%%%%%%%%%%%%%%%%%%%%%%%%%%%%%%%%%%%%%%%%%%%%%%%%%%%%%%%%

\section{The $L^2$-Liouville property of the Laplacian on $M \setminus \{p\}$} \label{Liouville property manifolds}
As mentioned in the introduction, the Laplacian on a complete manifold satisfies the 
$L^2$-Liouville property and is essentially self-adjoint. In this section, we consider the case
of a manifold with a point removed, which is thus not complete, and show that these properties may fail.

More specifically,  we consider the case of a manifold with a point removed and 
show that positivity of the bottom of the spectrum and the dimension of the manifold are the keys to the failure 
of the $L^2$-Liouville property.
In particular, we let $M\setminus \{p\}$ be a manifold of dimension 2 or 3 with a point $p$ removed and suppose that 
$\lm_0=\lm_0(\Delta_D)>0$ for the bottom of the spectrum of the Dirichlet Laplacian $\Delta_D$.
Then, the fact that the dimension is 2 or 3 will imply that the Green's function with pole at $p$ is in $L^2$ on a neighborhood of $p$
and $\lambda_0>0$ will imply the Green's function is in $L^2$ outside of a compact set which includes $p$, implying the 
failure of $L^2$-Liouville property of $M \setminus \{p\}$.
We note that the Laplacian on $M \setminus \{p\}$ 
is essentially self-adjoint if and only if $\mbox{dim } M \ge 4$, see \cite{CdV,M}.

We start by recalling some basic facts and proving an estimate for the Green's function in terms of capacity.
The capacity cap$(\Omega)$
of a precompact open set $\Omega \subseteq M$ is defined as
\[\textrm{cap}(\Omega) = \lim_{n\to\infty} \inf_{\psi \in \textup{Lip} (K_n, \Omega)}\|\nabla \psi\|^2_{L^2(K_n)} \]
where $(K_n)$ is an exhaustion of $M$ with $\overline{\Omega} \subseteq K_n$ and
$\textup{Lip}(K_n, \Omega) \subseteq H^1_0(M)$ is the set of locally Lipschitz functions $\psi$
on $M$ with compact support in $K_n$ 
such that $0 \le \psi \le 1$ and $\psi|_{\overline \Omega}=1$ \cite{G}.
We note that this notion is distinct from the capacity of the Cauchy boundary introduced in the previous section.
We recall that $\inf_{\psi \in \textup{Lip} (K_n, \Omega)}\|\nabla \psi\|^2_{L^2(K_n)}$ is attained by 
the equilibrium potential $u_n$ of $\Omega$ in $K_n$ which is 
the unique solution to
\[
\begin{cases}
\Delta u_n (x)=0, & x \in K_n \setminus \overline{\Omega}\\
u_n(x) =0, & x \in M \setminus K_n\\
u_n(x)=1, & x \in \overline\Omega.
\end{cases}
\]

\begin{lemma}\label{lemma5}
Let $M$ be a connected Riemannian manifold $M$ without boundary.
Let $p \in M$ and let $\Omega \subset M$ be a precompact open set
with $p \in \Omega$.
Let $\lambda_0 = \lambda_0(\Delta_D)$ denote the bottom of the spectrum of the Laplacian $\Delta_D$.
If $\lambda_0>0 $, then $M$ admits a positive Green's function $G$ and
\[
\|G(p, \cdot)\|_{L^2(M \setminus \Omega,\mu)} \le 
{C \over \sqrt{ \lambda_0}}\sqrt{ \textnormal{cap}(\Omega)}
\]
where $C = \sup_{y \in \partial \Omega} G(x,y)$.
\end{lemma}
\begin{proof}
Let $\Omega \subset M$ be a precompact open set with $x \in \Omega$ and let $(K_n)$ be an exhaustion of $M$
such that $\Omega \subset K_1$.
Since $\lambda_0>0$,  $M$ is transient and thus Proposition 10.1 of \cite{G} implies that
$M$ admits a positive Green's function $G$ and
\[\mbox{cap}(\Omega)>0.\]
For $p \in \Omega$, set $g_n(\cdot)=G_n(p,\cdot)$ and $g(\cdot)= G(p,\cdot)$, 
where $G_n$ is the Green's function of $K_n$ with Dirichlet boundary conditions
extended by 0 to $M \setminus K_n$.
Let 
\[C= \sup_{y \in \partial \Omega} g(y).\]
For $n>1$, $g_n \le g$ by the maximum principle.
Let $u_n \in H^1_0(M)$ be the equilibrium potential of $\Omega$ in $K_n$.
Since $g_n \le Cu_n$ on $\partial ( K_n \setminus \Omega )$,
it follows that $g_n \le Cu_n$ on $ K_n \setminus \Omega$ for every $n>1$ by the maximum principle.
Hence
\begin{align*} \|g_n\|_{L^2(K_n \setminus \Omega)} &\le 
C\|u_n\|_{L^2(K_n \setminus \Omega)} \\
&\le 
{C\over \sqrt{\lambda_0}} \|\nabla u_n\|_{L^2(K_n \setminus \Omega)}  \to 
{C\over \sqrt{\lambda_0}}  \sqrt{\mbox{cap}(\Omega)}
\end{align*}
as $n\to\infty$.
Since $0 \le g_n\le g_{n+1}$ and $g_n \to g$ as $n \to \infty$ on $M \setminus \Omega$, 
we apply Beppo Levi's monotone convergence theorem to get
\[
\|g\|_{L^2(M \setminus \Omega)} \le 
{C \over \sqrt{\lambda_0}} \sqrt{\mbox{cap}(\Omega)}
\]
which completes the proof.
\end{proof}

\begin{remark}
That $\lm_0>0$ implies the Green's function is in $L^2$ outside
of a compact set can also be obtained via decay estimates on the Green's function on a complete manifold, see
Corollary~22.4 in \cite{Li}. Our approach is more elementary as we only use the capacity of a set
and the associated equilibrium potential.
Furthermore, we do not assume completeness of the starting manifold.
\end{remark}

The result above gives that the Green's function is in $L^2$
outside of a compact set provided that the bottom of the spectrum is strictly positive. 
We now combine this with the fact that the Green's function is in $L^2$ on this compact
set if the dimension is small to show the failure of the $L^2$-Liouville property for manifolds
with positive bottom of the spectrum and small dimension.

\begin{theorem}\label{t:NoLiouv_manifolds}
Let $M$ be a connected Riemannian manifold without boundary.
Let $\lm_0 = \lm_0(\Delta_D)$ denote the bottom of the spectrum of the Laplacian $\Delta_D$.
If the dimension of $M$ is 2 or 3 and $\lambda_0>0$, then $G(p,\cdot)$, 
the positive Green's function with pole at $p \in M$, is in  $L^2 (M \setminus \{p\}, \mu)$. 
In particular, the $L^2$-Liouville property fails on
$M \setminus \{p\}$ for any $p \in M$. 
\end{theorem}
\begin{proof} Let $p \in M$ and let $r>0$ be such that $B_r(p)$ is a precompact open set.
If $\lambda_0>0$, then  
the positive Green's function $G$ with pole at $p$ exists and satisfies
 $G(p,\cdot) \in L^2(M \setminus B_r(p))$ by Lemma \ref{lemma5}.
This together with the standard fact that $G(p,\cdot) \in L^2(B_r(p) \setminus \{p\})$
if the dimension of $M$ is 2 or 3, see \cite{G}, allows us to conclude that $G(p,\cdot) \in L^2(M \setminus \{p\})$.
\end{proof}

In contrast, in the rest of the section we will show that if the dimension is 4 or greater, then the $L^2$-Liouville property
always holds on a manifold with a point removed.

\begin{theorem}\label{922theorem7}
Let $M$ be a connected complete Riemannian manifold $M$ without boundary.
If the dimension of $M$ is at least 4, then the $L^2$-Liouville property holds on $M \setminus \{p\}$ 
 for any $p \in M$. 
\end{theorem}

We prove the theorem through a series of lemmas.
For a fixed point $p \in M$, we write $r(x)=d(p,x)$ for the distance from $x \in M$ to $p$. 
For the following two lemmas, we assume
that we are on an $n$-dimensional connected complete Riemannian manifold without boundary.
\begin{lemma} \label{922lemma8}
For a harmonic function $u$ on $B_1(p) \setminus \{p\}$, if
\[u(x) =
\begin{cases}
o \left( {1\over r(x)^{n-2}} \right), & n\ge 3 \\
o \left( - \log (r(x)) \right), & n=2
\end{cases} 
\]
as $r(x) \to 0$,
then $u$  can be extended to $B_1(p)$ on which $u$ is harmonic.
\end{lemma}
\begin{proof}
This is standard.
One can apply the same argument as in the proof of Theorem 1.28 in \cite{HL}.
\end{proof}
\begin{lemma}\label{922lemma9}
For a harmonic function $u$ on $B_1(p) \setminus \{p\}$, if $u \in L^2(B_1(p) \setminus \{p\})$,
then
\[u(x) = o \left( {1 \over r(x)^{n/2}} \right) \textup{ as } r(x) \to 0.\]
\end{lemma}
\begin{proof} Note that the Ricci curvature is uniformly bounded on $B_1(p)$. 
Thus, there exists $r_0>0$ such that for any $r < r_0$, the mean value inequality holds for $u$ \cite{LS}.
Therefore, for any $x \in \partial B_r(p)$,
\[
u^2(x) \le {C \over \textrm{vol} \left( B_{r/2}(x) \right)} \int_{B_{r/2}(x)} u^2 d\mu
\le {C \over r^n} \int_{B_{2r}(p) \setminus \{p\}} u^2  d\mu = o \left( {1 \over r^n} \right) \textup{ as } r \to 0
\]
where the last assertion follows from the assumption that $u \in L^2(B_1(p) \setminus \{p\})$. 
This proves the lemma.
\end{proof}

We now combine the above lemmas to prove our second theorem of this section.
\begin{proof}[Proof of Theorem \ref{922theorem7}]
Let $u$ be a harmonic function on $M \setminus \{p\}$ such that $u \in L^2(M \setminus \{p\})$.
By Lemma \ref{922lemma9}, 
\[  u(x) = o \left( {1 \over r(x)^{n/2}} \right) \textup{ as } r(x) \to 0.\]
For $n \ge 4$, this yields
\[  u(x) = o \left( {1 \over r(x)^{n-2}} \right) \textup{ as } r(x) \to 0.\]
Hence, $u$ can be extended to a harmonic function on $B_1(p)$ by  Lemma \ref{922lemma8}
and thus to an $L^2$ harmonic function on $M$. By the fact that the manifold is complete, $u$ is constant \cite{Y}.
This proves the theorem.
\end{proof}

\begin{example}
The Laplacian on the hyperbolic space $\mathbb{H}^n$ has $\lambda_0>0$ and 
does not satisfy the $L^2$-Liouville property on $\mathbb{H}^n \setminus \{p\}$ for $n=2,3$ by Theorem \ref{t:NoLiouv_manifolds}.
Thus, we see that the dimension assumption in Theorem~\ref{922theorem7} is necessary.
In contrast, the Laplacian on $\mathbb R^n \setminus \{0\}$ satisfies the $L^2$-Liouville property for any $n \ge 1$, see \cite{ABR}.
\end{example}

\begin{remark}
By combining Corollary \ref{manifold_corollary} and 
the fact that for a complete manifold $M$ with no boundary the Laplacian on $M \setminus \{p\}$
is essentially self-adjoint if and only if 
dim $M \ge 4$, see \cite{M}, we can prove both:
\begin{itemize}
\item the failure of $L^2$-Liouville property under the assumptions of Theorem \ref{t:NoLiouv_manifolds}
if the manifold additionally has infinite measure,
\item the result found in Theorem \ref{922theorem7}. 
\end{itemize}
\end{remark}

%%%%%%%%%%%%%%%%%%%%%%%%%%%%%%%%%%%%%%%%%%%%%%%%%%%%%%%%%%%%%%

\section{The $\ell^2$-Liouville property of the Laplacian on $X \setminus \{o\}$}\label{s:Greens_graphs}
In the previous section, we considered the case of a manifold with a point removed. 
In particular, we showed that the dimension
is key to the failure of the $L^2$-Liouville property
when the manifold is made metrically incomplete via the removal of a point. The main tool to show this
was an estimate on the Green's function in terms of the capacity in the case of positive bottom of the spectrum.

In this section, we follow a similar development for graphs. However, we note that
neither dimension nor what to do following the removal of a point are well-established ideas for graphs.
In particular, as we do not assume local finiteness, we note that the removal of a vertex
can result in a graph with infinitely many connected components. The idea we follow here is that, at the
neighbors of the removed vertex, we attach infinite paths which have finite length, thus mimicking the 
manifold setting where the removal of a point results in an incomplete manifold.

We start with the definition of a capacity for a set. We note that this is distinct from the capacity
introduced in Section~\ref{s:examples} in the context of the Cauchy boundary of a graph. 
Let $b$ be a graph over $(X,m)$.
For a finite set $\Om \subseteq X$, we let
$$\cp(\Om) = \inf_{\vp \in C_c(X), \varphi _{\vert \Om} =1} \ms{Q}(\varphi)$$
where $\ms{Q}$ is the energy form of $b$.

We first show that for a finite set $\Om \subseteq X$, the capacity can be achieved via an exhaustion
technique. More specifically, we call a sequence of finite subsets $(K_n)$ of $X$ an exhaustion
sequence of $X$ if $b$ restricted to $K_n$ gives a connected graph, $K_n \subseteq K_{n+1}$
for all $n \in \N$ and $X = \bigcup K_n$. With these notions we state the following result which should
be known to experts. However, we briefly sketch the proof for the convenience of the reader.
\begin{lemma}\label{l:cap_exhaustion}
Let $b$ be a connected graph over $(X,m)$. 
If $\Om \subseteq X$ is finite and $(K_n)$ is an exhaustion sequence of $X$
with $\Om \subseteq K_1$, then
$$\cp(\Om) = \lim_{n \to \infty} \ms{Q}(\varphi_n)$$
where $\varphi_n \in C_c(X)$ satisfies
$$\begin{cases}
\ms{L} \varphi_n(x) =0, & x \in K_n \setminus \Om \\
\vp_n(x)=1, & x \in \Om \\
 \vp_n(x) = 0, & x \in X \setminus K_n \\
\end{cases}$$
for $n \in \N$.
\end{lemma}
\begin{proof}
For every $n \in \N$, we let
$$C(K_n, \Om) = \{ \vp \in C(K_n) \ | \ \vp_{\vert \Om}=1 \}$$
and extend any function in $C(K_n, \Om)$ by zero so that the function is defined on $X$.
As $C(K_n, \Om)$ is a finite dimensional space,
the energy form $\ms{Q}$ has a unique minimizer on $C(K_n, \Om)$ which we denote
by $\vp_n$ and which, via the Green's formula,
will satisfy the system of equations given in the statement. Clearly, we have
$\ms{Q}(\vp_{n+1}) \leq \ms{Q}(\vp_n)$ for all $n \in \N$.

A maximum principle for harmonic functions such as Lemma 1.39 in \cite{G18}
gives uniqueness of $\vp_n$ and 
will also imply
$$ 0 \leq \vp_n \leq 1 \qquad \textup{ and } \qquad  \vp_n \leq \vp_{n+1}$$
for all $n \in \N$. We note, in particular, that this can be used to establish the independence
of the limiting energy on the exhaustion sequence. Furthermore, for every $n \in \N$, we clearly have
$\cp(\Om) \leq \ms{Q}(\vp_n)$
so that 
$$\cp(\Om) \leq \liminf_{n \to \infty} \ms{Q}(\vp_n).$$ 

On the other hand, for every $\ve>0$,
there exists $\vp_\ve \in C_c(X)$ such that ${\vp_\ve}_{\vert \Om}=1$ and
$$\ms{Q}(\vp_\ve) \leq \cp(\Om) +\ve$$
Now, since $\vp_\ve \in C_c(X)$, there exists some $N$ such that $\vp_\ve$ is supported on $K_N$ and thus
$$\ms{Q}(\vp_\ve) \geq \ms{Q}(\vp_N) \geq \ms{Q}(\vp_n)$$
for all $n \geq N$. Therefore, for every $\ve>0$, we have
$$ \limsup_{n \to \infty}\ms{Q}(\vp_n) \leq \cp(\Om)+\ve$$
Combining inequalities and letting $\ve \to 0$ gives the result.
\end{proof}

The function $\vp_n$ constructed in Lemma~\ref{l:cap_exhaustion} above
is called the equilibrium potential of $\Om$ in $K_n$. These functions will be used below
to estimate the Green's function which will be introduced next.

We let $e^{-tL}$ denote the heat semigroup associated
to the Laplacian $L$ acting on $\ell^2(X,m)$. We then let $p: X \times X \times [0,\infty) \longrightarrow \R$
denote the heat kernel which is defined via
$$e^{-t L}f(x)=\sum_{y \in X}p_t(x,y) f(y)m(y)$$
for all $x \in X$, $t \geq 0$ and $f \in \ell^2(X,m)$. 
Finally, we let $G:X\times X \longrightarrow [0,\infty]$ denote the Green's
function defined via
$$G(x,y) = \int_0^\infty p_t(x,y) dt$$
for $x, y \in X$.

If $b$ over $(X,m)$ is connected and $G(x,y) < \infty$ for one pair of vertices $x, y \in X$, then we have
$G(x,y) < \infty$ for all pairs $x,y \in X$. In this case, we say that the graph is transient, otherwise we say that the graph
is recurrent. 
We note that the Green's function on a graph, and thus transience and recurrence, 
are often defined for random walks with a discrete
time parameter. However, this is equivalent to using continuous time, see \cite{Sch}.
By standard theory, transience is equivalent to the fact that $\cp(x)>0$
for some (equivalently, all) $x \in X$. Furthermore, whenever the bottom of the spectrum is positive, i.e.,
$$\lambda_0(L) = \inf_{\vp \in C_c(X)} \frac{\ms{Q}(\vp)}{\| \vp \|^2}>0,$$ the graph is transient,
see \cite{Sch, Soa, W, Woe} for further details.

We assume that $b$ is transient,
fix $o \in X$ and let $g \in C(X)$ be defined by $g(x)=G(o, x)$ for $x \in X$.
We then calculate that
$$\ms{L}g=\oh{1}_{o}$$
where $\oh{1}_{o} = 1_{o}/m(o)$ and $1_{o}$ is the characteristic function of the set $\{o\}$.
In particular, $g$ is a function which is harmonic everywhere except for a single vertex. As in the previous section,
we now discuss under which conditions $g$ is additionally in the Hilbert space $\ell^2(X,m)$. We then 
give a procedure for creating a harmonic function using $g$ on a graph over the vertex set $X \setminus \{o\}$.

In the following, we will need the standard fact that for any exhaustion sequence $(K_n)$, if we denote
the Laplacian resulting from restricting $\ms{L}$ to $\ell^2(K_n, m)$ by $L_n$ and the resulting heat
kernel by $p_t^n(x,y)$ for all vertices $x, y \in X$ and $t \geq0$, then
$$p_t^n(x,y) \to p_t(x,y)$$
as $n \to \infty$, see \cite{KL}. In particular, this implies
$$g_n(x) \to g(x)$$
for all $x \in X$ as $n \to \infty$ where $g_n(x) = \int_0^\infty p_t^n(o, x) dt$.
We note that for any $x \not \in K_n$, we have $p_t^n(o,x) =0 $ and thus $g_n(x)=0$.

For a set $\Om \subseteq X$, we let 
$$\partial \Om = \{ x \in \Om \ | \ \exists y \sim x, y \not \in \Om\}$$
denote the boundary of the set.
For a graph with positive bottom of the spectrum, we now show that the Green's function
is in $\ell^2(X,m)$. We note that, in contrast to the manifold case above, the Green's function
does not have a singularity at $o$. 

\begin{lemma}\label{l:greens_l^2}
Let $b$ be a connected graph over $(X,m)$.
Let $o \in X$ and let $\Om \subseteq X$ be finite with $o \in \Om$. Let $g(x) = G(o, x)$ for $x \in X$
and $C= \sup_{x \in \partial \Om} g(x)$ where $G$ denotes the Green's function.
Let $\lambda_0 = \lambda_0(L)$ denote the bottom of the spectrum of the Laplacian $L$.
If $\lambda_0 >0$, then 
$$\|g\|_{\ell^2(X \setminus \Om, m)} \leq \frac{C}{\sqrt{\lambda_0}} \sqrt{\cp(\Om)}.$$
In particular, $g \in \ell^2(X,m)$.
\end{lemma}

\begin{proof}
Let $(K_n)$ be an exhaustion sequence of $X$ such that $\Om \subseteq K_1$. Let $\vp_n$
denote the equilibrium potential of $\Om$ in $K_n$ discussed in Lemma~\ref{l:cap_exhaustion} 
and let $g_n(x)=G_n(o,x)$ denote the Green's function
on $K_n$. As $g_n \leq g$, we note that $g_n \leq C\vp_n$ on $\partial \Om$
and since $\ms{L} \vp_n = \ms{L} g_n = 0$ on $K_n \setminus \Om$ and $\vp_n = g_n =0$ outside of $K_n$,
it follows that
$$g_n \leq C\vp_n$$ 
by the maximum principle for harmonic functions, see
Lemma 1.39 in \cite{G18}. Therefore, using the definition of $\lambda_0$ and Lemma~\ref{l:cap_exhaustion}
we obtain
\begin{align*}
\|g_n\|^2_{\ell^2(K_n \setminus \Om)} &\leq C^2 \|\vp_n\|^2_{\ell^2(K_n \setminus \Om)} \leq \frac{C^2}{\lambda_0}\ms{Q}(\vp_n) \to  \frac{C^2}{\lambda_0} \cp(\Om)
\end{align*}
as $n \to \infty$. As $0 \le g_n\le g_{n+1}$ and $g_n \to g$ as $n \to \infty$ on $X$, 
we apply Beppo Levi's monotone convergence theorem to get
$$\|g\|_{\ell^2(X \setminus \Om)} \leq  \frac{C}{\sqrt{\lambda_0}} \sqrt{\cp(\Om)}.$$
The final statement follows since $\Om$ is finite and $g$ is defined everywhere.
\end{proof}

Our aim is now to mimic Theorem~\ref{t:NoLiouv_manifolds} from the manifold case and
to apply Lemma~\ref{l:greens_l^2} to create examples where the $\ell^2$-Liouville property fails.
As $g$ satisfies $\ms{L}g=\oh{1}_{o}$ the basic idea is to create a new graph by removing
the vertex $o$ where $g$ fails to be harmonic. For every vertex that was adjacent to $o$ in the
original graph and for which the value of $g$ at that vertex is different than the value of $g$ at $o$, 
we then replace that vertex with a path to infinity on which we extend $g$ to be harmonic on the path. 
We note that, in general, this process will not result in a connected graph.

In the construction outlined above, we will attach paths to infinity to certain neighbors of the removed vertex.
We first analyze under which condition the extension of a non-constant harmonic function
to such a path will be in $\ell^2(X,m)$. 
We will apply this to the function based on the Green's function
in what follows below.
\begin{lemma}\label{l:harmonic_extension}
Let $X = \N_0=\{0, 1, 2, \ldots\}$ and let $b$ be a graph over $(X,m)$ with $b(j,k)>0$ if and only if $|j-k|=1$.
If $v \in C(X)$ satisfies
$$\ms{L}v(r)=0$$
for $r \geq 1$, then $v$ is uniquely determined by the choice of
$v(1)$ and $v(0)$. More specifically, if $C=b(0,1)(v(1)-v(0))$, then 
$$v(r+1) = v(1)+ C \sum_{k=1}^r \frac{1}{b(k,k+1)}$$
for $r \geq 1$.
In particular, if $C \neq 0$, then $v$ is non-constant
and if 
$$\sum_{r=1}^\infty \left( \sum_{k=1}^{r} \frac{1}{b(k,k+1)} \right)^2 m(r+1) < \infty,$$
then $v \in \ell^2(X,m)$.
\end{lemma}
\begin{proof}
Let $v \in C(X)$ satisfy $\ms{L}v(r)=0$ for $r \geq 1$. Using induction, we see that
if $C=b(0,1)(v(1)-v(0))$, then we obtain 
$$v(r+1)-v(r) = \frac{C}{b(r,r+1)}$$
for all $r \geq 1$ so that
$$v(r+1) = v(1)+ C \sum_{k=1}^r \frac{1}{b(k,k+1)}.$$
From this formula, it is clear that if $C \neq 0$, then $v$
is non-constant and the summability assumption
implies that $v \in \ell^2(X,m)$ by straightforward estimates.
\end{proof}

\begin{remark}[A comment on the condition in Lemma~\ref{l:harmonic_extension}]
We comment on the summability condition $\sum_{r=1}^\infty \left( \sum_{k=1}^{r} 1/b(k,k+1) \right)^2 m(r+1) < \infty$
appearing above in Lemma~\ref{l:harmonic_extension}. As already noted, we wish to use this condition in order
to extend the Green's function to the path in such a way that the Green's function remains in $\ell^2(X,m)$.
As Lemma~\ref{l:harmonic_extension} shows, the condition on $b$ and $m$
allows the extension of a non-constant positive harmonic function to this path. 

Let us now contrast this with other conditions found for graphs with $X=\N$ and $b(j,k)>0$ if and only if $|j-k|=1$.
We note that by a direct calculation any function which is harmonic at all vertices on such a graph
is always constant. In particular, the Laplacian on all such graphs satisfies the $\ell^2$-Liouville property.
For this reason, in Lemma~\ref{l:harmonic_extension} we assume that the function is not harmonic
at the first vertex. We note if $m(\N)=\infty$, then the Laplacian on such a graph is essentially self-adjoint \cite{KL}
so from Corollary~\ref{c:graphs_esa_liouv} 
one would not expect to have non-constant square integrable harmonic functions in this case.
For $b$, the condition $\sum_k 1/b(k,k+1)<\infty$ is known to be equivalent to the 
transience of such graphs, see \cite{Woe}. Furthermore, in the case of $m(\N)<\infty$, transience, stochastic incompleteness
and the failure of form uniqueness are known to be equivalent, see \cite{GHKLW, Sch}. As the failure of form uniqueness
implies the failure of essential self-adjointness, in the case of $m(\N)<\infty$ and $\sum_k 1/b(k,k+1)<\infty$,
we have the failure of essential self-adjointness. Obviously, these two conditions imply the summability
assumption in Lemma~\ref{l:harmonic_extension} which
intertwines $b$ and $m$. In fact, it turns out that this condition
is equivalent to the failure of essential self-adjointness for the Laplacian on such graphs, see the forthcoming work \cite{IMW}.

Finally, let us note that the condition above actually forces the path to have finite length in any intrinsic path metric.
Thus, as in the manifold case, after the removal of the vertex and the addition of such a path, we end up 
with a space that is not metrically complete. We discuss these notions further below, see the discussion leading up to
and the statement of Corollary~\ref{c:incomplete}.
\end{remark}

We now make precise a way to construct a graph following the removal of a vertex. 
We let $b$ be a transient graph over $(X,m)$, 
let $g(x) = G(o,x)$ where $G$ is the Green's function
of the graph and $o\in X$. We then decompose the set of neighbors of $o$ into
$$N_{o}=\{x \in X \ | \ x \sim o, g(x) \neq g(o) \} \quad \textup{and} \quad N_{c} = \{x \in X \ | \ x \sim o, g(x) = g(o) \}.$$
We ultimately will attach paths to infinity to vertices in $N_o$ and a single additional vertex to
vertices in $N_c$. Then, the Green's function can be extended by Lemma~\ref{l:harmonic_extension}
to the paths which are attached to vertices in $N_o$ and by a constant to the new vertices
which are neighbors of vertices in $N_c$.

We note that $\ms{L}g(o) = 1/m(o)$ implies that $N_o$
is non-empty.
On the other hand, $N_o$ in general will not contain all of the neighbors of $o$. 
For example, if $x \sim o$ is a vertex such that all vertices that have a non-repeating path
that connects them to $o$ must contain $x$ and this set is finite,
then $x \in N_c$.

Now, for every $x \in N_o$, we let 
$\N_x=\{x_1, x_2, x_3, \ldots\}$ denote a copy of the natural numbers and for
every $x \in N_c$, we let $x_c$ denote a single new vertex. 
We then let $X_e = \bigcup_{x \in N_o} \N_x$ and $X_c = \bigcup_{x \in N_c} \{x_c\}$ and
define a new vertex set via
$$X_o = X\setminus \{o\} \cup X_e \cup X_c.$$
We then let $b_o: X_o \times X_o \longrightarrow [0, \infty)$ be an edge weight defined so as to be symmetric and
satisfy:
\begin{itemize}
\item $b_o(x,y) = b(x,y)$ for $x,y \in X \setminus \{o\}$,
\item for every $x \in N_o$ we let $b_o(x,x_1)>0$ for $x_1 \in \N_x$,
\item for $x \in N_o$ and $x_j, x_k \in \N_x$
we let $b_o(x_j,x_k) > 0$ if and only if $|j-k|=1$,
\item for every $x \in N_c$ we let $b_o(x,x_c)>0$ for $x_c \in X_c$,
\item $b_o$ be $0$ for all other pairs of vertices. 
\end{itemize}
In other words, we remove the vertex $o$ and
to every neighbor of $o$ in $N_o$ we now attach an infinite path and to every
neighbor of $o$ in $N_c$ we attach a single vertex and
 no other connections. The measure
$m$ defined on $X$ can be extended to $X_o$ in an arbitrary manner; however, we will specify
some additional requirements for both $m_o$ and $b_o$ below.

\begin{remark}[A comment on connectedness]
We recall that a path is a sequence of vertices $(x_n)$ with $b(x_n, x_{n+1}) > 0$
for all $n \in \N$ and that we call a graph connected if for any two vertices
there exists a path that starts at one of the vertices and ends at the other. 
We note that the graph $b_o$ over $(X_o,m_o)$ will not, in general, be connected.
This deserves some comment as we wish to analyze the $\ell^2$-Liouville property and the
essential self-adjointness of the Laplacian associated to $b_o$ and the connections between
these properties have only been established for connected graphs. We note that, a graph that 
is not connected will, in general, not satisfy the $\ell^2$-Liouville property as a harmonic function
can take on different constant values on the connected components, i.e., the maximal connected subsets
of the vertex set. In particular, if the removal of a vertex results in a graph with at least two connected
components of finite measure, then the Laplacian on the resulting graph will never satisfy the $\ell^2$-Liouville 
property. On the other hand, the Laplacian on such a graph might still be essentially self-adjoint. 
However, for our results below, we attach paths to infinity
and show that there is at least one infinite connected component where both of these properties fail
for some additional assumptions of the edge weights and measure.
In particular, the Laplacian can be decomposed into a direct sum of operators on connected components
and if the operator on one of these components is not essentially self-adjoint, then the 
Laplacian on the entire graph is not essentially self-adjoint.
\end{remark}

Given this new graph $b_o$ over $(X_o,m_o)$, Lemma~\ref{l:harmonic_extension}
shows that if $b$ is transient, then there is a unique way of using the values of $g$, 
the Green's function of $b$ over $(X,m)$,
to define a function $g_o$ on the new vertex set $X_o$ so that $g_o$
is harmonic. 
More specifically, we let $g_o$ be a function on $X_o$ defined by letting
$g_o(x) = g(x)$ for $x \in X \setminus \{o\}$ and $g_o(x_c)=g(o)$ for
every $x_c \sim x \in N_c$.
For every $x \in N_o$ and $x_1 \in \N_x$ we first let
$$g_o(x_1)=g(x) + \frac{1}{b_o(x,x_1)} \sum_{y \in X \setminus \{o\}} b(x,y)(g(x)-g(y)).$$
We note that since $x \in N_o$ we have $g(x) \neq g(o)$ and since $\ms{L}g(x)=0$
we obtain $\sum_{y \in X \setminus \{o\}} b(x,y)(g(x)-g(y)) \neq 0$.
In particular, we see that
$g_o(x_1) \neq g(x)$.
We then use Lemma~\ref{l:harmonic_extension} to define $g_o(x_r)$ for $r > 1$
so that $g_o$ is harmonic on $\N_x$. 
More specifically, if $C= b_o(x_1,x)(g_o(x_1)-g(x))$, then Lemma~\ref{l:harmonic_extension} gives
$$g_o(x_{r+1}) = g_o(x_1) + C\sum_{n=1}^r \frac{1}{b_o(x_n, x_{n+1})}$$
for $r \geq 1$.
As $C \neq 0$, we see that $g_o$ is not constant on $\N_x$.

In this way, it follows by a direct calculation that $g_o$ is harmonic
on $X_o$. 
Furthermore, if $g \in \ell^2(X,m)$, then Lemma~\ref{l:harmonic_extension} shows
how to choose the weights $b_o$ and measure $m_o$ in such a way to ensure that the
extension $g_o \in \ell^2(X_o,m_o)$. More specifically, we obtain the following statement
which can be thought as a counterpart to Theorem~\ref{t:NoLiouv_manifolds} from the manifold setting.

\begin{theorem}\label{t:graph_no_Liouv}
Let $b$ be a connected graph over $(X,m)$ such that $\ms{L}(C_c(X)) \subseteq \ell^2(X,m)$
and $\lambda_0>0$. 
Let $o \in X$ be such that the number of neighbors of $o$ is finite. 
Let $b_o$ over $(X_o,m_o)$ be defined as above and so that for every
$x \in N_o$ and for $x_n \in \N_x$ we have
$$\sum_{r=1}^\infty \left( \sum_{n=1}^{r}\frac{1}{b_o(x_n,x_{n+1})}\right)^2 m_o(x_{r+1})<\infty.$$
Then, there exists a harmonic function in $\ell^2(X_o,m_o)$ which is non-constant
on an infinite connected component of $b_o$.
In particular, the Laplacian $L_{o,c}$ associated to $b_o$ over $(X_o,m_o)$ does not satisfy
the $\ell^2$-Liouville property and is not essentially self-adjoint. 
\end{theorem}
\begin{proof}
By Lemma~\ref{l:greens_l^2}, as we assume that $\lm_0>0$, we have $g \in \ell^2(X,m)$. 
By Lemma~\ref{l:harmonic_extension}
we obtain that $g_o \in \ell^2(\N_x, m_o)$ for every $x \in N_o$ by our summability 
assumptions. As both $N_o$ and $N_c$ are finite sets by assumption,
we obtain that $g_o \in \ell^2(X_o, m_o)$. Therefore, $g_o$ is a non-constant harmonic
function in $\ell^2(X_o,m_o)$ so that the $\ell^2$-Liouville property fails for $b_o$ over $(X_o,m_o)$. 
Furthermore, as $N_o \neq \emptyset$ it follows from the construction that there exists
at least one infinite connected component of $b_o$ over $(X_o,m_o)$ for which
$g_o$ is non-constant. In particular, for this infinite connected component, it follows
that Corollary~\ref{c:graphs_esa_liouv} applies to the restriction of $L_{o,c}$ to this component,
and thus this restriction is not essentially self-adjoint. From this it follows
that $L_{o,c}$ is not essentially self-adjoint.
\end{proof}

We note that some of the results on manifolds from the 
previous section invoke completeness in various parts. 
We will now 
make precise an analogue to geodesic completeness for graphs.
In particular, we wish to discuss how in Theorem~\ref{t:graph_no_Liouv}
the graph resulting from the removal of a vertex and attachment of paths as specified
cannot be geodesically complete.

A natural way to define a metric on a connected graph is to assign lengths to edges
and then take the length of the shortest path that connects two vertices. 
We let $\si:X \times X \longrightarrow [0,\infty)$ be a symmetric function such that
$$\si(x,y) > 0 \qquad \textup{ if and only if} \qquad x \sim y.$$
We think of $\si(x,y)$ as the length of the edge connecting $x$ and $y$ and thus
call $\si$ a length function.
For vertices $x, y \in X$ we let $\Pi_{x,y}$ denote the set of all paths from $x$ to $y$, that is,
$(x_k)_{k=0}^n \in \Pi_{x,y}$ if $(x_k)$ is a path with $x_0=x$ and $x_n =y$. We then let the length
of a path $(x_k)$ be defined by $l_\si((x_k))= \sum_{k=0}^{n-1} \si(x_k, x_{k+1})$ and define 
a metric via
$$d_{\si}(x,y)= \inf_{(x_k) \in \Pi_{x,y}} l_\si((x_k)).$$
We call any metric arising in this way a path metric.
We call a path $(x_n)$ a geodesic if $d_{\si}(x_j,x_k) = \sum_{i=j}^{k-1} \si(x_i, x_{i+1})$
for all indices $j<k$ for which the path is defined. We call the graph geodesically complete
if every infinite geodesic has infinite length, that is, $l_{\si}((x_n))=\infty$
if $(x_n)$ is a geodesic with infinitely many vertices.

A graph is called locally finite if every vertex has only finitely many neighbors.
Theorem A.1 in \cite{HKMW} states that for path metrics on locally finite graphs
geodesic completeness, metric completeness and the fact that all balls defined with respect
to the metric are finite are equivalent. Thus, for locally finite graphs we call such graphs
complete with respect to the path metric. However, to derive essential self-adjointness from
completeness, we need to consider metrics that also take into account the values of the 
edge weights as well as the vertex measure. These are the so-called intrinsic metrics which 
we introduce next. 

Specifically, we call a metric $\vr$ on $X$ intrinsic for $b$ over $(X,m)$ if
$$\sum_{y \in X} b(x,y) \vr^2(x,y) \leq m(x)$$
for all $x \in X$.
A natural way to construct an intrinsic metric using lengths is to let $\Deg(x) = \frac{1}{m(x)} \sum_{y \in X} b(x,y)$ denote the
weighted degree of a vertex $x \in X$ and let
$$\si(x,y) = \min \left \{\frac{1}{\sqrt{\Deg(x)}},\frac{1}{\sqrt{\Deg(y)}} \right\}$$
denote the length of an edge. Letting $\vr_\si = d_\si$ denote the
path metric defined via $\si$ as above, then as
$\vr_\si(x,y) \leq \si(x,y)$ for all $x,y$ with $b(x,y)>0$ we get that $\vr_\si$ is intrinsic.
Thus, intrinsic metrics always exist in the graph setting. However, in contrast to the case for manifolds,
it is not true that there exists a unique maximal intrinsic metric. For more background on the notion and use
of intrinsic metrics in our setting, see \cite{FLW, K, Woj2}.

One of the main results of \cite{HKMW} states that, for locally finite graphs and path metrics,
completeness with respect to an intrinsic path metric implies essential self-adjointness. Thus we
directly derive the following.
\begin{corollary}\label{c:incomplete}
Let $b$ be a connected locally finite graph over $(X,m)$ with $\lambda_0>0$. 
Let $b_o$ over $(X_o,m_o)$ be defined so that for every
$x \in N_o$ and $x_n \in \N_x$ we have
$$\sum_{r=1}^\infty \left( \sum_{n=1}^{r}\frac{1}{b_o(x_n,x_{n+1})}\right)^2 m_o(x_{r+1})<\infty.$$
Then, $b_o$ over $(X_o,m_o)$ is not complete with respect to any intrinsic path metric. In particular,
$$l_\si((x_n))<\infty$$ 
for all paths attached at $x \in N_o$ where $\si$ is a length function which defines an intrinsic path metric
for $b_o$ over $(X_o,m_o)$.
\end{corollary}
\begin{proof}
The first part of the statement follows 
from Theorem~\ref{t:graph_no_Liouv} above and Theorem~2 in \cite{HKMW} which gives that
the failure of essential self-adjointness implies incompleteness with respect to intrinsic path metrics. In particular,
we note that the local finiteness assumption implies that $o$ has only finitely many neighbors
and also that $\mathcal{L}(C_c(X)) \subseteq \ell^2(X,m)$ as follows by a direct calculation.
For the statement on the length of paths, if $l_\si((x_n))=\infty$ for some path satisfying
the summability assumption, 
then we could start with a complete graph with $\lambda_0>0$ and 
carrying out the procedure to construct $b_o$ over $(X_o,m_o)$ using this path
would then result in a complete
graph as we would attach only paths to infinity which have infinite length. Thus, the resulting graph
would not have any geodesic of finite length and would thus be complete giving a contradiction to what we
have just shown.
Thus, we obtain that $l_\si((x_n))<\infty$ for all paths satisfying the summability assumption
and for all lengths $\si$ which define intrinsic path metrics on $b_o$ over $(X_o,m_o)$.
This completes the proof.
\end{proof}

We finish the paper with two remarks.

\begin{remark}
The assumption that $\lm_0>0$ can be removed in the corollary above. 
In particular, we note that both the fact that $\si$ defines an intrinsic path metric for $b_o$ over $(\N_x,m_o)$
and the summability assumption on the edge weights and measure on the path $\N_x$
are independent of the original graph $b$ over $(X,m)$. Thus, we can use the result above
to show that all such paths must have finite length with respect to an intrinsic metric regardless
of the starting graph. Then, attaching such a path to any graph will then give metric incompleteness
if $\si$ is extended to the graph in such a way as to make an intrinsic path metric. We note that we
have assumed transience of $b$ over $(X,m)$ 
for the construction of $b_o$ over $(X_o,m_o)$; however, this is also not necessary for this result
as attaching paths of finite length will always produce an incomplete graph.
\end{remark}

\begin{remark}
We recall that Theorem~\ref{922theorem7} in the manifold setting gives that if a manifold
is complete and the dimension of the manifold is greater than or equal to 4, 
then removing a point results in a manifold which satisfies the $L^2$-Liouville property.
In fact, by \cite{M}, it follows in this case that the Laplacian on the manifold is essentially self-adjoint.

An analogue to this for the graph case is that starting with a locally finite complete graph, by removing
a point and attaching paths which have infinite length with respect to a length which gives an intrinsic metric,
the resulting graph is still complete. In particular, the Laplacian
on such a graph is essentially self-adjoint and, if connected, satisfies the $\ell^2$-Liouville property.

In terms of capacity, there is a capacity which is positive for points in Euclidean space for dimensions 2 and 3,
while zero for points in dimensions $4$ and above, see \cite{HKM}. 
The graph case also has a counterpart in the sense that when adding paths to infinity, we may introduce
a boundary point at infinity when such paths have finite length and then consider the capacity of such a point.
In this context, it would be interesting to see if there is a condition
that would force either positive capacity or capacity zero of a boundary point at infinity for every intrinsic path metric.
\end{remark}

\section*{Acknowledgements}
The authors are grateful to Marcel Schmidt for a careful reading of the manuscript
and for helpful comments. R.K.W.~would also like to thank J{\'o}zef Dodziuk for his support
and for many fruitful discussions. Furthermore, the authors would like to thank the referees
for useful comments and Isaac Pesenson for the invitation to submit this article.


\begin{bibdiv}
\begin{biblist}

\bib{ABR}{book}{
   author={Axler, Sheldon},
   author={Bourdon, Paul},
   author={Ramey, Wade},
   title={Harmonic function theory},
   series={Graduate Texts in Mathematics},
   volume={137},
   edition={2},
   publisher={Springer-Verlag, New York},
   date={2001},
   pages={xii+259},
   isbn={0-387-95218-7},
   review={\MR{1805196}},
   doi={10.1007/978-1-4757-8137-3},
}
	

\bib{BP}{article}{
   author={Boscain, Ugo},
   author={Prandi, Dario},
   title={Self-adjoint extensions and stochastic completeness of the
   Laplace-Beltrami operator on conic and anticonic surfaces},
   journal={J. Differential Equations},
   volume={260},
   date={2016},
   number={4},
   pages={3234--3269},
   issn={0022-0396},
   review={\MR{3434398}},
   doi={10.1016/j.jde.2015.10.011},
}

\bib{C}{article}{
   author={Chernoff, Paul R.},
   title={Essential self-adjointness of powers of generators of hyperbolic
   equations},
   journal={J. Functional Analysis},
   volume={12},
   date={1973},
   pages={401--414},
   review={\MR{0369890 (51 \#6119)}},
}

\bib{CL}{article}{
   author={Chen, Roger},
   author={Li, Peter},
   title={On Poincar\'{e} type inequalities},
   journal={Trans. Amer. Math. Soc.},
   volume={349},
   date={1997},
   number={4},
   pages={1561--1585},
   issn={0002-9947},
   review={\MR{1401517}},
   doi={10.1090/S0002-9947-97-01813-8},
}

\bib{CdV}{article}{
   author={Colin de Verdi\`ere, Yves},
   title={Pseudo-laplaciens. I},
   language={French, with English summary},
   journal={Ann. Inst. Fourier (Grenoble)},
   volume={32},
   date={1982},
   number={3},
   pages={xiii, 275--286},
   issn={0373-0956},
   review={\MR{688031}},
}


\bib{FLW}{article}{
   author={Frank, Rupert L.},
   author={Lenz, Daniel},
   author={Wingert, Daniel},
   title={Intrinsic metrics for non-local symmetric Dirichlet forms and
   applications to spectral theory},
   journal={J. Funct. Anal.},
   volume={266},
   date={2014},
   number={8},
   pages={4765--4808},
   issn={0022-1236},
   review={\MR{3177322}},
   doi={10.1016/j.jfa.2014.02.008},
}


\bib{G1}{article}{
   author={Gaffney, Matthew P.},
   title={The harmonic operator for exterior differential forms},
   journal={Proc. Nat. Acad. Sci. U. S. A.},
   volume={37},
   date={1951},
   pages={48--50},
   issn={0027-8424},
   review={\MR{0048138 (13,987b)}},
}

\bib{G2}{article}{
   author={Gaffney, Matthew P.},
   title={A special Stokes's theorem for complete Riemannian manifolds},
   journal={Ann. of Math. (2)},
   volume={60},
   date={1954},
   pages={140--145},
   issn={0003-486X},
   review={\MR{0062490 (15,986d)}},
}

\bib{GHKLW}{article}{
   author={Georgakopoulos, Agelos},
   author={Haeseler, Sebastian},
   author={Keller, Matthias},
   author={Lenz, Daniel},
   author={Wojciechowski, Rados{\l}aw K.},
   title={Graphs of finite measure},
   journal={J. Math. Pures Appl. (9)},
   volume={103},
   date={2015},
   number={5},
   pages={1093--1131},
   issn={0021-7824},
   review={\MR{3333051}},
   doi={10.1016/j.matpur.2014.10.006},
}

\bib{GM}{article}{
   author={Gesztesy, Fritz},
   author={Mitrea, Marius},
   title={A description of all self-adjoint extensions of the Laplacian and
   Kre\u{\i}n-type resolvent formulas on non-smooth domains},
   journal={J. Anal. Math.},
   volume={113},
   date={2011},
   pages={53--172},
   issn={0021-7670},
   review={\MR{2788354}},
   doi={10.1007/s11854-011-0002-2},
}


\bib{G}{article}{
   author={Grigor{\cprime}yan, Alexander},
   title={Analytic and geometric background of recurrence and non-explosion
   of the Brownian motion on Riemannian manifolds},
   journal={Bull. Amer. Math. Soc. (N.S.)},
   volume={36},
   date={1999},
   number={2},
   pages={135--249},
   issn={0273-0979},
   review={\MR{1659871 (99k:58195)}},
   doi={10.1090/S0273-0979-99-00776-4},
}

\bib{G18}{book}{
   author={Grigor{\cprime}yan, Alexander},
   title={Introduction to analysis on graphs},
   series={University Lecture Series},
   volume={71},
   publisher={American Mathematical Society, Providence, RI},
   date={2018},
   pages={viii+150},
   isbn={978-1-4704-4397-9},
   review={\MR{3822363}},
   doi={10.1090/ulect/071},
}

\bib{G19}{book}{
   author={Grigor{\cprime}yan, Alexander},
   title={Heat kernel and analysis on manifolds},
   series={AMS/IP Studies in Advanced Mathematics},
   volume={47},
   publisher={American Mathematical Society, Providence, RI; International
   Press, Boston, MA},
   date={2009},
   pages={xviii+482},
   isbn={978-0-8218-4935-4},
   review={\MR{2569498 (2011e:58041)}},
}


\bib{HKLW}{article}{
   author={Haeseler, Sebastian},
   author={Keller, Matthias},
   author={Lenz, Daniel},
   author={Wojciechowski, Rados{\l}aw},
   title={Laplacians on infinite graphs: Dirichlet and Neumann boundary
   conditions},
   journal={J. Spectr. Theory},
   volume={2},
   date={2012},
   number={4},
   pages={397--432},
   issn={1664-039X},
   review={\MR{2947294}},
}




\bib{HL}{book}{
   author={Han, Qing},
   author={Lin, Fanghua},
   title={Elliptic partial differential equations},
   series={Courant Lecture Notes in Mathematics},
   volume={1},
   edition={2},
   publisher={Courant Institute of Mathematical Sciences, New York; American
   Mathematical Society, Providence, RI},
   date={2011},
   pages={x+147},
   isbn={978-0-8218-5313-9},
   review={\MR{2777537}},
}


\bib{HKM}{article}{
   author={Hinz, Michael},
   author={Kang, Seunghyun},
   author={Masamune, Jun},
   title={Probabilistic characterizations of essential self-adjointness and
   removability of singularities},
   language={English, with English and Russian summaries},
   journal={Mat. Fiz. Komp\cprime yut. Model.},
   date={2017},
   number={3(40)},
   pages={148--162},
   issn={2587-6325},
   review={\MR{3706135}},
   doi={10.15688/mpcm.jvolsu.2017.3.11},
}


\bib{HK}{article}{
   author={Hua, Bobo},
   author={Keller, Matthias},
   title={Harmonic functions of general graph Laplacians},
   journal={Calc. Var. Partial Differential Equations},
   volume={51},
   date={2014},
   number={1-2},
   pages={343--362},
   issn={0944-2669},
   review={\MR{3247392}},
   doi={10.1007/s00526-013-0677-6},
}

\bib{HKMW}{article}{
   author={Huang, Xueping},
   author={Keller, Matthias},
   author={Masamune, Jun},
   author={Wojciechowski, Rados{\l}aw K.},
   title={A note on self-adjoint extensions of the Laplacian on weighted
   graphs},
   journal={J. Funct. Anal.},
   volume={265},
   date={2013},
   number={8},
   pages={1556--1578},
   issn={0022-1236},
   review={\MR{3079229}},
   doi={10.1016/j.jfa.2013.06.004},
}

\bib{I}{article}{
   author={Inoue, Atsushi},
   title={Essential self-adjointness of Schr{\"o}dinger operators on the weighted integers},
   journal={forthcoming},
}


\bib{IMW}{article}{
   author={Inoue, Atsushi},
   author={Masamune, Jun},
   author={Wojciechowski, Rados{\l}aw K.},
   title={Essential self-adjointness of the Laplacian on weighted graphs: stability and characterizations},
   journal={forthcoming},
}

\bib{K}{article}{
   author={Keller, Matthias},
   title={Intrinsic metrics on graphs: a survey},
   conference={
      title={Mathematical technology of networks},
   },
   book={
      series={Springer Proc. Math. Stat.},
      volume={128},
      publisher={Springer, Cham},
   },
   date={2015},
   pages={81--119},
   review={\MR{3375157}},
%   doi={10.1007/978-3-319-16619-3_7},
}


\bib{KL}{article}{
   author={Keller, Matthias},
   author={Lenz, Daniel},
   title={Dirichlet forms and stochastic completeness of graphs and
   subgraphs},
   journal={J. Reine Angew. Math.},
   volume={666},
   date={2012},
   pages={189--223},
   issn={0075-4102},
   review={\MR{2920886}},
  doi={10.1515/CRELLE.2011.122},
}

\bib{KLW}{article}{
   author={Keller, Matthias},
   author={Lenz, Daniel},
   author={Wojciechowski, Rados{\l}aw K.},
   title={Volume growth, spectrum and stochastic completeness of infinite
   graphs},
   journal={Math. Z.},
   volume={274},
   date={2013},
   number={3-4},
   pages={905--932},
   issn={0025-5874},
   review={\MR{3078252}},
   doi={10.1007/s00209-012-1101-1},
}

\bib{Li}{book}{
   author={Li, Peter},
   title={Geometric analysis},
   series={Cambridge Studies in Advanced Mathematics},
   volume={134},
   publisher={Cambridge University Press, Cambridge},
   date={2012},
   pages={x+406},
   isbn={978-1-107-02064-1},
   review={\MR{2962229}},
   doi={10.1017/CBO9781139105798},
}


\bib{LS}{article}{
   author={Li, Peter},
   author={Schoen, Richard},
   title={$L^p$ and mean value properties of subharmonic functions on
   Riemannian manifolds},
   journal={Acta Math.},
   volume={153},
   date={1984},
   number={3-4},
   pages={279--301},
   issn={0001-5962},
   review={\MR{766266}},
   doi={10.1007/BF02392380},
}


\bib{M}{article}{
   author={Masamune, Jun},
   title={Essential self-adjointness of Laplacians on Riemannian manifolds
   with fractal boundary},
   journal={Comm. Partial Differential Equations},
   volume={24},
   date={1999},
   number={3-4},
   pages={749--757},
   issn={0360-5302},
   review={\MR{1683058}},
   doi={10.1080/03605309908821442},
}
	
		

\bib{RS1}{book}{
   author={Reed, Michael},
   author={Simon, Barry},
   title={Methods of modern mathematical physics. I},
   edition={2},
   note={Functional analysis},
   publisher={Academic Press, Inc. [Harcourt Brace Jovanovich, Publishers],
   New York},
   date={1980},
   pages={xv+400},
   isbn={0-12-585050-6},
   review={\MR{751959}},
}


\bib{RS2}{book}{
   author={Reed, Michael},
   author={Simon, Barry},
   title={Methods of modern mathematical physics. II. Fourier analysis,
   self-adjointness},
   publisher={Academic Press [Harcourt Brace Jovanovich, Publishers], New
   York-London},
   date={1975},
   pages={xv+361},
   review={\MR{0493420}},
}


\bib{Sch}{article}{
   author={Schmidt, Marcel},
   title={Global properties of Dirichlet forms on discrete spaces},
   journal={Dissertationes Math.},
   volume={522},
   date={2017},
   pages={43},
   issn={0012-3862},
   review={\MR{3649359}},
   doi={10.4064/dm738-7-2016},
}

\bib{Sch2}{article}{
   author={Schmidt, Marcel},
   title={On the existence and uniqueness of self-adjoint realizations of discrete (magnetic) Schr{\"o}dinger operators},
   conference={
      title={Analysis and geometry on graphs and manifolds},
   },
   book={
      series={London Math. Soc. Lecture Note Ser.},
      volume={461},
      publisher={Cambridge Univ. Press, Cambridge},
   },
   date={2020},
%   pages={250--327},
%   review={\MR{3966604}},
%eprint={arXiv:1805.08446 [math.FA]}
}


\bib{Soa}{book}{
   author={Soardi, Paolo M.},
   title={Potential theory on infinite networks},
   series={Lecture Notes in Mathematics},
   volume={1590},
   publisher={Springer-Verlag},
   place={Berlin},
   date={1994},
   pages={viii+187},
   isbn={3-540-58448-X},
   review={\MR{1324344 (96i:31005)}},
}

\bib{S}{article}{
   author={Strichartz, Robert S.},
   title={Analysis of the Laplacian on the complete Riemannian manifold},
   journal={J. Funct. Anal.},
   volume={52},
   date={1983},
   number={1},
   pages={48--79},
   issn={0022-1236},
   review={\MR{705991}},
   doi={10.1016/0022-1236(83)90090-3},
}

\bib{W}{book}{
   author={Woess, Wolfgang},
   title={Random walks on infinite graphs and groups},
   series={Cambridge Tracts in Mathematics},
   volume={138},
   publisher={Cambridge University Press},
   place={Cambridge},
   date={2000},
   pages={xii+334},
   isbn={0-521-55292-3},
   review={\MR{1743100 (2001k:60006)}},
   doi={10.1017/CBO9780511470967},
}

\bib{Woe}{book}{
   author={Woess, Wolfgang},
   title={Denumerable Markov chains},
   series={EMS Textbooks in Mathematics},
   note={Generating functions, boundary theory, random walks on trees},
   publisher={European Mathematical Society (EMS), Z\"urich},
   date={2009},
   pages={xviii+351},
   isbn={978-3-03719-071-5},
   review={\MR{2548569 (2011f:60142)}},
   doi={10.4171/071},
}


\bib{Woj}{book}{
   author={Wojciechowski, Radoslaw Krzysztof},
   title={Stochastic completeness of graphs},
   note={Thesis (Ph.D.)--City University of New York},
   publisher={ProQuest LLC, Ann Arbor, MI},
   date={2008},
   pages={87},
   isbn={978-0549-58579-4},
   review={\MR{2711706}},
}

\bib{Woj2}{article}{
   author={Wojciechowski, Rados{\l}aw K.},
   title={Stochastic completeness of graphs: bounded Laplacians, intrinsic metrics, volume growth and curvature},
   journal={J. Fourier Anal. Appl.},
%   volume={25},
   date={to appear},
   eprint={arXiv:2010.02009 [math.MG]},
%   number={7},
%   pages={659--670},
 %  issn={0022-2518},
  % review={\MR{0417452}}
}



\bib{Y}{article}{
   author={Yau, Shing Tung},
   title={Some function-theoretic properties of complete Riemannian manifold
   and their applications to geometry},
   journal={Indiana Univ. Math. J.},
   volume={25},
   date={1976},
   number={7},
   pages={659--670},
   issn={0022-2518},
   review={\MR{0417452}},
}


\end{biblist}
\end{bibdiv}
\end{document}